\documentclass[11pt]{article}
\usepackage[applemac]{inputenc}
\usepackage[T1]{fontenc}
\usepackage{lmodern,bbm}
\usepackage{amsmath}
\usepackage{amssymb}
\usepackage{amsfonts}
\usepackage{mathrsfs}
\usepackage{amsthm}
\usepackage{stmaryrd}
\usepackage{esint}

\usepackage[colorlinks]{hyperref}
\hypersetup{
    colorlinks=true,
    linkcolor=blue,
    citecolor=blue,      
}

\theoremstyle{theorem}
\newtheorem{theo}{Theorem}[section]
\newtheorem{lemma}[theo]{Lemma}

\theoremstyle{definition}
\newtheorem{definition}[theo]{Definition}

\theoremstyle{remark}
\newtheorem{rem}[theo]{Remark}

\numberwithin{equation}{section}

\DeclareMathOperator\R{\mathbb{R}}

\DeclareMathOperator{\Tr}{Tr}
\DeclareMathOperator{\id}{id}
\newcommand{\norm}[1]{\left\lVert#1\right\rVert}
\newcommand*\prob{\mathop{}\!\mathscr{P}}
\newcommand*\diff{\mathop{}\!\mathrm{d}}

\newcommand{\mres}{\mathbin{\vrule height 1.6ex depth 0pt width
0.13ex\vrule height 0.13ex depth 0pt width 1.3ex}} 

\begin{document}
\author{Thibault Caillet\thanks{\scriptsize Institut Camille Jordan, Universit\'e Claude Bernard - Lyon 1; 43 boulevard du 11 novembre 1918, 69622 Villeurbanne cedex
\texttt{caillet@math.univ-lyon1.fr }.}}
\title{The five gradients inequality for non quadratic costs}
\maketitle

\begin{abstract} We give a proof of the "five gradients inequality" of Optimal Transportation Theory for general costs of the form $c(x,y)=h(x-y)$ where $h$ is a $C^1$ strictly convex radially symmetric function. \end{abstract}

\section{Introduction}

While its name was popularized later on, the five gradients inequality was introduced in \cite{BVEstimates} as a way to derive estimates on the gradient of some recurring variational problems in Optimal Transport involving the Wasserstein distance $W_2$. In particular the authors showed that the inequality can be used in the celebrated JKO scheme
\begin{equation*}
\label{varproblem}
    \varrho^\tau _{k+1} \in \text{argmin}_\varrho  \int f(\varrho)\diff x + \frac{W^2 _2 (\varrho, \varrho^\tau _{k})}{2 \tau},
\end{equation*}
to derive BV estimates that can be iterated along the scheme uniformly in $\tau$, and therefore pass to the limit PDE  $\partial_t \varrho - \nabla \cdot \left( \varrho f''(\varrho) \nabla \varrho \right) =0$, yielding that the BV norm of the solution is nonincreasing in time. In the same paper, the authors also use the inequality to prove BV estimates for the Wasserstein projection of a measure with
BV density on the set of measures with density bounded by another given BV function. This result is then used in \cite{mutuallysingular} to find bounds on the perimeters of solutions of some variational problems involving mutually singular measures. The five gradients inequality has also been used in \cite{FokkerPlanckLp} to derive Sobolev estimates for the solutions of the JKO scheme for the Fokker-Planck equation. One only needs to prove that the five gradients inequality holds in a more general setting to generalize most of these results beyond the $W_2$ case. It is now folk-lore that the inequality is also true for the distance $W_p$ for $p>1$ and more general cost functions, yet a full proof has not been available until now. In this paper we generalize the proof given in \cite{BVEstimates} to the case where the cost $c$ is of the form $c(x,y) = h(x-y)$ where $h$ is a strictly convex radially symmetric $C^1$ function. The inequality reads as follows:
\begin{theo}
\label{5grads}
Let $\Omega \subset \R^d$ be bounded and convex with non-empty interior, $\varrho,g \in W^{1,1}(\Omega)$ be two probability densities, $h\in C^1(\R^d)$ a radially symmetric strictly convex function and $H \in C^1(\R^d \backslash \{0\})$ be a radially symmetric convex function, then the following inequality holds 
\begin{equation}
    \int_{\Omega} \big( \nabla \varrho \cdot \nabla H ( \nabla \varphi) + \nabla g \cdot \nabla H (\nabla \psi) \big) \diff x \geq 0,
\end{equation}
where $(\varphi,\psi)$ is a choice of Kantorovich potentials for the optimal transport problem between $\varrho$ and $g$ for the transport cost given by $h$, with the convention that $\nabla H (0) = 0$.
\end{theo}
A particular case is the one where $h(z)=|z|^p$ for $p>1$, generalizing the inequality to the $W_p$ case. Following the strategy established in \cite{BVEstimates}, the generalized inequality for example implies that the BV norm of the solution of $\partial_t \varrho -\Delta_q  (g(\varrho))=0$ (with $q=\frac{p}{p-1}$, and $g$ nondecreasing) decreases in time. Indeed, this nonlinear PDE can be seen as the limit of a JKO-like scheme 
\begin{equation*}
    \varrho_{k+1} \in \textnormal{argmin}_\varrho \frac{W^p _p (\varrho, \varrho_k)}{p \tau^{p-1}} + \int_\Omega f(\varrho_k) \diff x.
\end{equation*}
by choosing $g'(\varrho)=\varrho^{p-1} f''(\varrho)$ with suitable assumptions on $f$ so that the JKO scheme indeed converges (see e.g. \cite{Agueh}, \cite{AmbrosioGigliSavare},  \cite{Otto}).

To justify computations that involve second derivatives of Kantorovich potentials, the original proof made use of the well known Caffarelli regularity theory available for the cost $h(x-y)=|x-y|^2$. Since the works of Ma, Trudinger, Wang \cite{MTW} and Loeper \cite{Loeper}, sufficient and necessary conditions on the cost $c$ to guarantee the existence of smooth potentials are known, and in these cases one could reproduce the proof given in \cite{BVEstimates}. Unfornutately, these conditions for regularity do not cover the cases where, for example, $h(z)=|z|^p$ for $p\neq 2$. Therefore in the sequel we shall instead approximate the cost with semiconcave cost functions, and use the fact that Kantovorich potentials inherit this semiconcavity, along with Alexandroff's theorem (see e.g. \cite{EvansGariepy}) :
\begin{theo}[Alexandroff's theorem]
\label{Alexandroff}
Let $u$ be a semiconcave function on an open bounded set $A \subset \R^d$; then $u$ is twice differentiable a.e., meaning for a.e. $x_0 \in A$, there exists $p_{x_0} \in \R ^d $ and a symmetric matrix $B_{x_0}$ such that
\begin{equation*}
    \lim_{x\xrightarrow{} x_0} \frac{u(x) - u(x_0) - p_{x_0} \cdot (x-x_0) +B_{x_0}(x-x_0) \cdot (x-x_0)}{|x-x_0|^2} =0.
\end{equation*}
Moreover, the gradient of $u$, defined for almost every $x_0$ in $A$ is BV and the absolutely continuous part of the second derivative $D^2 _{ac} u(x_0)$ is given by $B_{x_0}$.
\end{theo}
Usual references for the theory of optimal transport that we will use throughout this paper include \cite{VillaniOT}, \cite{VillaniO&N} as well as \cite{OTAM} which also features a chapter dedicated to the JKO scheme for the Fokker-Planck equation.

\section{Proof of the inequality}
In the sequel, unless otherwise indicated, $\Omega$ will denote a bounded convex subset of $\R^d$ with non-empty interior. The weak convergence of measures will be in duality with $C(\Bar{\Omega})$, however, we will work with probability measures that have densities and that therefore cannot be concentrated on $\partial \Omega$ which is negligible for the Lebesgue measure because $\Omega$ is convex.
We take $R>0$ to be such that $\Omega \subset \Bar{B}(R/2)$.
\begin{definition}
We say that a cost $c(x,y)=h(x-y)$ satisfies (H1) if :
\begin{enumerate}
    \item $h \in C^2(\Bar{B}(R))$
    \item $h$ is strictly convex
    \item $h$ is radially symmetric
\end{enumerate}
\end{definition}
\begin{rem}
In particular, a cost satisfying (H1) is semiconcave i.e. there exists $C >0$ such that $x \mapsto h(x)-C|x|^2$ is concave on $B(R)$. Since it is known that Kantorovich potentials can be taken to be $c$-concave, i.e. of the form
\begin{equation*}
    \varphi(x) = \inf_{y \in \Omega} h(x-y) - \psi(y)
\end{equation*}
for some function $\psi : \Omega \to \Bar{\R}$, they can also be asumed to be semiconcave, with the same semiconcavity constant $C$ as $h$.
\end{rem}

In order to deal with regularity issues, we will approximate the cost function with costs satisfying (H1).
\begin{lemma}
\label{lemmacostconv}
Let $h\in C^1(\R^d)$ be a radially symmetric strictly convex function and $\eta _\varepsilon$ be a radially symmetric mollifier. Then $h_\varepsilon = \eta _\varepsilon * h$ satisfies \textnormal{(H1)} and $h_\varepsilon \xrightarrow{C^1(\Bar{B}(R))} h$.
\end{lemma}
To prove the convergence back to the original problem we will need a few lemmas. \break
We recall that if $(X, \Sigma, \mu)$ is a measure space and $f_n, f$ are measurable functions, we say that $f_n$ converges in $\mu$-measure to $f$ if for every $\varepsilon>0$
\begin{equation*}
    \lim_{n \xrightarrow{}+\infty} \mu \left( \left\{ \left| f_n -f \right|> \varepsilon \right\} \right) = 0.
\end{equation*}
and we will denote this convergence by $ f_n \xrightarrow{\mu} f$.

\begin{lemma}
\label{lemmacvgmeasure}
Let $\Omega \subset \R^d$ and $T,T_n: \Omega \xrightarrow{} \Omega$ be measurable functions. Let $\varrho \in \prob (\Omega)$ then we have
\begin{equation*}
    T_n \xrightarrow{\varrho} T \iff (\id \times T_n)_{\#}\varrho \xrightarrow{}(\id \times T)_{\#}\varrho \textnormal{ weakly},
\end{equation*}
where $(\textnormal{id}\times T ): \Omega \to \Omega$ is the function defined as $(\textnormal{id}\times T_n) (x)= (x,T(x))$
\end{lemma}
\begin{proof}
First if $T_n \xrightarrow{\varrho} T$, assuming by contradiction that the weak convergence does not hold, there exists $f \in C_b(\Omega \times \Omega)$, $\varepsilon >0$ and a subsequence $(T_{n_k})_k$ such that 
\begin{equation*}
    \left|\int_\Omega f(x,T(x))- f(x,T_{n_k}(x)) \diff \varrho(x) \, \right| \geq \varepsilon.
\end{equation*}
Extracting a further subsequence such that $T_{{n_k}_j}$ converges to $T$ $\varrho$-a.e. and using dominated convergences gives a contradiction. \\
Let $\varepsilon >0$, by Lusin's theorem, for $\delta >0$, there exists a compact set $K \subset \Omega$ such that $\left.T\right|_K$ is continuous and $\varrho(\Omega \backslash K) \leq \delta$. Therefore $A= \{ (x,y)\in K\times \Omega, |T(x) - y| \geq \varepsilon \}$ is a closed set in $\Omega \times \Omega$. Since the evaluation on closed sets is upper semi continuous for the weak convergence of measures,
\begin{align*}
    0=(\id \times T)_{\#}\varrho(A) \geq & \limsup (\id \times T_n)_{\#}\varrho(A) \\
    =& \limsup \varrho \big( \{x \in K, |T(x) - T_n(x)| \geq \varepsilon \} \big) \\
    \geq & \limsup \varrho \big( \{x \in \Omega, |T(x) - T_n(x)| \geq \varepsilon \} \big) - \delta.
\end{align*}
Letting $\delta \xrightarrow{}0$ gives $T_n \xrightarrow{\varrho} T$.
\end{proof}

\begin{lemma}
\label{cvginverse}
Given metric spaces $X$ and $Y$, let $f,f_n : X \xrightarrow{}Y$ be bijective functions such that $f_n$ converges uniformly to $f$. Then $g_n=f_n^{-1}$ converges uniformly to $g=f^{-1}$ if $g$ is uniformly continuous. Furthermore, if $(x_n)_n \in X$ is a sequence such that $f_n(x_n) \xrightarrow{}f(x)$ then $x_n \xrightarrow{}x$.
\end{lemma}

\begin{proof}
It is straightforward to check that if $w : Y \xrightarrow{}X $ is uniformly continuous and if $h_n : X \xrightarrow{}Y$ converges uniformly to $h$, then $w \circ h_n$ converges uniformly to $w \circ h$. Therefore to prove that $g_n = g \circ f \circ g_n$ uniformly converges to $g \circ f \circ g$ it is enough to check that $f \circ g_n $ uniformly converges to the identity function. For $y \in Y$, we have that $d\left(f \circ g_n (y), y\right)= d(f \circ g_n(y), f_n \circ g_n (y))$ and the uniform convergence of $f_n$ to $f$ concludes. \\
The second claim follows from the fact that uniform convergence of functions preserves the convergence of sequences of points.
\end{proof}

\begin{lemma}
\label{transportcvgpp}
Let $T_n , T : \Omega \to \R^d$ be uniformly bounded, $1 < p<\infty$ and $\varrho_n , \varrho \in L^1 (\Omega)$ such that $\varrho_n \xrightarrow{L^1(\Omega)} \varrho$, $\varrho_n T_n \rightharpoonup \varrho T $ as measures in $\mathcal{M}(\Omega)$ and $\int_\Omega |T_n|^p \varrho_n \xrightarrow{} \int_\Omega |T|^p \varrho$. Then $T_n \xrightarrow{L^p(\Omega,\varrho)} T$.
\end{lemma}

\begin{rem}
The assumptions are satisfied if $\Omega$ is bounded and $T_n,T$ take values in $\Omega$, $\varrho_n \xrightarrow{L^1(\Omega)} \varrho$ and $(\id \times T_n)_{\#} \varrho_n \rightharpoonup (\id \times T)_{\#} \varrho$
\end{rem}

\begin{proof}
Since the $T_n$ are uniformly bounded in $L^\infty$, up to a subsequence we have weak-$*$ convergence in $L^\infty$ : $T_n \overset{\ast}{\rightharpoonup} \tilde{T}\in L^\infty(\Omega)$. Since $\varrho_n \xrightarrow{} \varrho$ strongly in $L^1$, strong-weak convergence yields convergence in distribution of the product $\varrho_n T_n \xrightarrow{} \varrho \tilde{T}$. By uniqueness of the limit in the space of distributions, the convergence $\varrho_n T_n \rightharpoonup \varrho T$ therefore gives that that $\varrho T = \varrho \tilde{T}$. Then we have :
\begin{align*}
    \left| \int_\Omega |T_n|^p \varrho - \int_\Omega |T|^p \varrho \right| &\leq \left| \int_\Omega |T_n|^p \varrho - \int_\Omega |T_n|^p \varrho_n \right| + \left| \int_\Omega |T_n|^p \varrho_n - \int_\Omega |T|^p \varrho \right| .
\end{align*}
The second term goes to $0$ by assumption and the first term is bounded by $C^p \norm{\varrho - \varrho_n}_1$ which also goes to $0$. Now it is enough to prove weak convergence in $L^p$, so let $f \in L^{p'}(\Omega, \varrho)$. Since by H\"older $f\varrho \in L^1(\Omega)$,
\begin{align*}
    \left| \int_\Omega f T \varrho \diff x - \int_\Omega f T_n \varrho \diff x \right| = \left| \int_\Omega (T-T_n)f \varrho \diff x \right| \xrightarrow{n\xrightarrow{}\infty} \left| \int_\Omega (T-\tilde{T})f \varrho \diff x \right| =0 .
\end{align*}
Hence $T_n \xrightarrow{L^p(\Omega,\varrho)} T$, and by uniqueness of this limit the whole sequence actually converges.
\end{proof}

We are now ready to begin the proof of the five gradients inequality with smooth densities and a cost satisfying (H1)
\begin{lemma}
\label{lemme5grad}
Given $h$ satisfying (H1), $\varrho$ and $g$ smooth probability densities, let $H \in C^2(\R^d)$ be a convex function, then the following inequality holds :
\begin{equation}
    \int_{\Omega} \big(  \varrho \, \nabla \cdot [\nabla H ( \nabla \varphi)] -  g \, \nabla \cdot [\nabla H (-\nabla \psi)] \big) \diff x \leq 0,
\end{equation}
where $(\varphi,\psi)$ is a choice of Kantorovich potentials for the optimal transport problem between $\varrho$ and $g$ for the transport cost given by $h$, and $\nabla \, \cdot$ denotes the distributional divergence.
\end{lemma}

\begin{proof}
Let $(\varphi,\psi)$ be a choice of c-concave potentials for the transport problem. From c-concavity we deduce that $\varphi$ and $\psi$ are semiconcave and hence by Alexandroff's theorem (\ref{Alexandroff}), $\varphi$ and $\psi$ are twice differentiable almost everywhere and $\nabla \varphi$, $\nabla \psi$ are functions of bounded variation on $\Omega$.\\
We will denote by $J_\varphi$ the set of approximate jump points of $\nabla\varphi$, $\nabla \varphi ^+$ and $\nabla \varphi ^-$ the left and right approximations of $\nabla \varphi$ and $\nu_\varphi$ the approximate normal to $J_\varphi$. We will denote by $D^2 _{ac} \varphi$ and $D^2 _{c} \varphi$ respectively the absolutely continuous part and the Cantor part of $D^2 \varphi$. We will now use the chain rule for BV functions (see \cite{ambrosio2000fbv} for more details on notation and precise statements).
\begin{equation*}
    D[\nabla H ( \nabla \varphi)] = D^2 H(\nabla\varphi) D^2 _{ac} \varphi \mathcal{L}^d + D^2 H(\nabla\varphi) D^2 _{c}\varphi + (\nabla H ( \nabla \varphi^+) - \nabla H ( \nabla \varphi^-))\otimes \nu_{\varphi}\mathcal{H}^{d-1} \mres J_\varphi .
\end{equation*}
Using the semiconcavity of $\varphi$, we deduce that there exists some $C>0$ such that, as a measure, $D^2 \varphi \leq C \mathcal{L}^d$, and we have 
\begin{align*}
    &D^2 _{ac} \varphi \leq C \; \; \mathcal{L}^d \text{ a.e.},\\
    &D^2 _{c} \varphi \leq 0 ,\\
    &(\nabla \varphi^+ - \nabla \varphi^-) \otimes \nu_{\varphi}\mathcal{H}^{d-1} \mres J_\varphi \leq 0.
\end{align*}
In particular since $a\otimes b \leq 0$ implies the existence of $\lambda \leq 0$ such that $a=\lambda b$, using the convexity of $H$ we get 
\begin{equation*}
    (\nabla H ( \nabla \varphi^+) - \nabla H ( \nabla \varphi^-)) \cdot (\nabla \varphi^+  - \nabla \varphi^-) \geq 0,
\end{equation*}
and therefore
\begin{equation*}
    (\nabla H ( \nabla \varphi^+) - \nabla H ( \nabla \varphi^-)) \cdot  \nu_\varphi \leq 0 \; \; \mathcal{H}^{d-1} \mres J_\varphi \text{ ae.}
\end{equation*}
Using the nonnegativity of $D^2 H$ we therefore have 
\begin{align*}
    \nabla \cdot [\nabla H ( \nabla \varphi)]  =& \; \Tr(D^2 H(\nabla\varphi) D^2 _{ac} \varphi) \\
    &+ \Tr(D^2 H(\nabla\varphi) D^2 _{c} \varphi)\\
    &+(\nabla H ( \nabla \varphi^+) - \nabla H ( \nabla \varphi^-)) \cdot  \nu_\varphi\mathcal{H}^{d-1} \mres J_\varphi \\
    &\leq \Tr(D^2 H(\nabla\varphi) D^2 _{ac} \varphi) .
\end{align*}
Using the same arguments we also get 
\begin{equation*}
    \nabla \cdot [\nabla H (- \nabla \psi)] \geq \Tr(D^2 H(-\nabla\psi) D^2 _{ac} \psi).
\end{equation*}
Integrating with respect to non negative densities $\varrho, g$, we get
\begin{align*}
    \int_{\Omega} \big(  \varrho \, \nabla \cdot [\nabla H & ( \nabla \varphi)] -  g \, \nabla \cdot [\nabla H (-\nabla \psi)] \big) \diff x \\
    &\leq \int_{\Omega} \big(  \varrho \Tr(D^2 H(\nabla\varphi) D^2 _{ac} \varphi) \,  -  g \, \Tr(D^2 H(-\nabla\psi) D^2 _{ac} \psi)  \big) \diff x \\
    &= \int_{\Omega} \big(  \varrho \Tr(D^2 H(\nabla\varphi) D^2 _{ac} \varphi) \,  -  \varrho \, \big [\Tr(D^2 H(-\nabla\psi) D^2 _{ac} \psi) \circ T \big] \big) \diff x ,
\end{align*}
where $T$ is the optimal transport map between $\varrho$ and $g$. Let $S$ be the optimal transport between $g$ and $\varrho$ (for the cost $\tilde{c}(x,y)= c(y,x)$), then $S\circ T= \textnormal{id}$ $\varrho$-a.e., and therefore for $\varrho$-a.e. $x$ we have
\begin{align*}
    -\nabla\psi(T(x))&=\nabla h ( S\circ T(x) - T(x))) \\
    &= \nabla h ( x - T(x))\\
    &= \nabla \varphi (x) .
\end{align*}
We know that $\varphi,\psi$ are twice differentiable $\mathcal{L}^d$-a.e. hence
\begin{align*}
    1 &=g\big(\{y, \psi \text{ is twice differentiable at }y \}\big)\\
    &=\varrho \big(\{x, \psi \text{ is twice differentiable at }T(x) \}\big) .
\end{align*}
and since the function $v \mapsto \varphi(x +v) + \psi(T(x)-v) \leq h(x-y)$ is maximal at $v=0$ for $\varrho$-a.e. $x$, using that, for BV functions, the absolutely continuous part of the derivative coincides with the pointwise derivative, we have $D^2 _{ac} \varphi(x) +D^2 _{ac}\psi(T(x)) \leq 0$ $\varrho$-a.e. Therefore

\begin{align*}
    \int_{\Omega} \big(  \varrho \, \nabla \cdot [\nabla H & ( \nabla \varphi)] -  g \, \nabla \cdot [\nabla H (-\nabla \psi)] \big) \diff x \\
    & \leq \int_{\Omega} \big(  \varrho \Tr(D^2 H(\nabla\varphi) [D^2 _{ac} \varphi + D^2 _{ac}\psi \circ T] \big) \diff x \leq 0. \qedhere
\end{align*}
\end{proof}

\begin{theo}
\label{th5gradsmooth}
Given $h$ satisfying (H1), suppose that $\varrho,g$ are smooth probability densities, and let $H \in C^2(\R^d)$ be a convex radially symmetric function, then the following inequality holds 
\begin{equation*}
    \int_{\Omega} \big( \nabla \varrho \cdot \nabla H ( \nabla \varphi) + \nabla g \cdot \nabla H (\nabla \psi) \big) \diff x \geq 0,
\end{equation*}
where $(\varphi,\psi)$ is a choice of Kantorovich potentials for the optimal transport problem between $\varrho$ and $g$ for the transport cost given by $h$.
\end{theo}

\begin{proof}
First since $H$ is radially symmetric we have $\nabla H(-\nabla \psi)=-\nabla H(\nabla \psi)$ Using Lemma \ref{lemme5grad} and the trace theorem for BV functions we get
\begin{align*}
    \int_{\Omega} \big( \nabla \varrho& \cdot \nabla H ( \nabla \varphi) + \nabla g \cdot \nabla H (\nabla \psi) \big) \diff x \\
    &= - \int_{\Omega} \big(  \varrho \, \nabla \cdot [\nabla H ( \nabla \varphi)] +  g \, \nabla \cdot [\nabla H (\nabla \psi)] \big) \diff x \\
    & \;  + \int_{\partial \Omega} \big( \varrho \nabla H(\nabla \varphi) \cdot n + g \nabla H (\nabla \psi) \cdot n \big) \diff \mathcal{H}^{d-1} \\
    & \geq  \int_{\partial \Omega} \big( \varrho \nabla H(\nabla \varphi) \cdot n + g \nabla H (\nabla \psi) \cdot n \big) \diff \mathcal{H}^{d-1}.
\end{align*}

Let $x_0 \in \partial \Omega$ be such that $\varrho (x_0) >0$. By continuity of $\varrho$, for small $\varepsilon>0$, there exists $r_0>0$ such that if $x \in B(x_0,r_0) \cap \Omega$, then $\varrho(x) > \varepsilon$. For a.e. $x \in B(x_0,r_0) \cap \Omega$, since $T(x) \in \Omega$ we have $(T(x) - x_0) \cdot n(x_0) \leq 0$ and therefore $(x-\nabla h^* (\nabla \varphi)) \cdot n(x_0) \leq x_0 \cdot n(x_0)$. For $r\leq r_0$, integrating on $B(x_0,r)$ we get
\begin{equation*}
    \fint_{B(x_0,r) \cap \Omega} (x-\nabla h^* (\nabla \varphi)) \cdot n(x_0) \diff x \leq \fint_{B(x_0,r) \cap \Omega} x_0 \cdot n(x_0) \diff x.
\end{equation*}
Taking $r \xrightarrow{} 0$ we get $\nabla h^* (\nabla \varphi(x_0)) \cdot n(x_0) \geq 0 $ for $\mathcal{H}^{d-1}$-a.e. $x_0 \in \partial \Omega$ such that $\varrho(x_0) > 0$. \\

Since $H$ and $h$ are convex and radially symmetric we have $\nabla H(z)=\alpha(z) z $ with $\alpha(z) \geq 0$ and since $\nabla h^* = (\nabla h)^{-1}$ we also have $\nabla h^* (z) = \alpha'(z) z$ hence for $\mathcal{H}^{d-1}$-a.e. $x_0 \in \partial \Omega$ such that $\varrho(x_0) > 0$ we have $\nabla H (\nabla \varphi(x_0)) \cdot n(x_0) \geq 0$. By the same arguments for $g$ and $\nabla \psi$ we have
\begin{equation*}
   \int_{\partial \Omega} \big( \varrho \nabla H(\nabla \varphi) \cdot n + g \nabla H (\nabla \psi) \cdot n \big) \diff \mathcal{H}^{d-1} \geq 0 ,
\end{equation*}
which concludes the proof.
\end{proof}

\begin{rem}
In the above proof the hypothesis that $h$ and $H$ are radially symmetric is only used in proving that $\nabla H$ and $\nabla h^*$ point in the same direction, to deal with the boundary integral. One therefore should be able to only assume that $H(z)=\alpha( h^* (z))$ for some increasing convex function $\alpha$ to get $\nabla H(z)=\alpha' (h^* (z))\nabla h^* (z)$. This however adds the difficulty of finding an adequate replacement for Lemma \ref{lemmacostconv} and proving the ensuing convergence.
\end{rem}

\textbf{Proof of Theorem \ref{5grads}.}

\begin{proof}
We begin by taking $H\in C^2(\R^d)$, smooth densitites $\varrho, g$, and approximating the cost $h$ using Lemma \ref{lemmacostconv} and applying Theorem \ref{th5gradsmooth}. Since $h_\varepsilon \xrightarrow{C^1} h$, by compactness of $\prob (\Omega)$ the sequence $(\gamma_\varepsilon)_\varepsilon =((\id \times T_\varepsilon)_{\#}\varrho)_\varepsilon$ of optimal plans weakly converges up to a subsequence to $\gamma = (\id \times T)_{\#}\varrho \in \prob (\Omega)$ which is optimal for $h$. By uniqueness of such transport maps in fact the whole sequence converges and by Lemma \ref{lemmacvgmeasure} we have $T_\varepsilon \xrightarrow{\varrho} T$. Extracting a subsequence we therefore have $T_\varepsilon \xrightarrow{} T$ $\varrho$-a.e, using that we have $T_\varepsilon (x)=x - \nabla h_\varepsilon ^*(\nabla \varphi_\varepsilon (x))$ and $T (x)=x - \nabla h ^*(\nabla \varphi(x))$ and Lemma \ref{cvginverse} we get that $\nabla \varphi_\varepsilon (x) \xrightarrow{} \nabla \varphi(x)$ $\varrho$-a.e. Doing the same for $g$ and $\psi$ and using that $\nabla H \in L^\infty$, the result is proved by dominated convergence. \\

Next we take $\varrho,g \in W^{1,1}(\Omega)$, and for example by convolution we can find smooth functions $\varrho_\varepsilon, g_\varepsilon$ such that $(\varrho_\varepsilon,g_\varepsilon) \xrightarrow{W^{1,1}(\Omega)}(\varrho,g)$. Extracting a subsequence if necessary, we can assume the convergence to hold a.e. and to have a uniform $L^1$ domination. Normalizing $(\varrho_\varepsilon,g_\varepsilon)$ if necessary we can assume they are probability densities and apply the previous result to get 

\begin{equation}
    \int_{\Omega} \big( \nabla \varrho_\varepsilon \cdot \nabla H ( \nabla \varphi_\varepsilon) + \nabla g_\varepsilon \cdot \nabla H (\nabla \psi_\varepsilon) \big) \diff x \geq 0.
\end{equation}
Using Lemma $\ref{transportcvgpp}$ and up to a subsequence, we have $T_\varepsilon \xrightarrow{}T$ $\varrho$-a.e. which implies just as before that $\nabla \varphi_\varepsilon \xrightarrow{} \nabla \varphi$ $\varrho$-a.e. Therefore for a.e. $x\in \{\varrho >0 \}$, we have $\nabla \varphi_\varepsilon (x) \xrightarrow{} \nabla \varphi (x) $ and, using $\nabla \varrho = 0$ a.e on $\left\{ \varrho =0\right\}$, for a.e. $x \in \{\varrho =0\}$, we have $\nabla \varrho_\varepsilon (x) \xrightarrow{} \nabla \varrho(x)$. In any case, since $\nabla H \in L^\infty$, we have that $\nabla \varrho_\varepsilon (x) \cdot \nabla H ( \nabla \varphi_\varepsilon (x)) \xrightarrow{} \nabla \varrho (x) \cdot \nabla H ( \nabla \varphi (x))$ for a.e. $x \in \Omega$ and dominated convergence gives
\begin{align*}
    \int_{\Omega} \ \nabla \varrho_\varepsilon \cdot \nabla H ( \nabla \varphi_\varepsilon) \diff x 
    \xrightarrow{\varepsilon \xrightarrow{}0} \int_{\Omega}  \  \nabla \varrho \cdot \nabla H ( \nabla \varphi)  \diff x.
\end{align*}
\noindent
Doing the same for $g$ and $\psi$ and adding the integrals gives the result.\\
Finally, approximating a radially symmetric convex function $H \in C^1(\R^d \backslash \{0\})$ with $C^2$ functions with the same properties proves that the result stays true in this case, when using the natural convention $\nabla H(0) = 0$.
\end{proof}

\end{document}